\documentclass[oneside,english]{amsart}
\usepackage[T1]{fontenc}
\usepackage[latin9]{inputenc}
\usepackage{verbatim}
\usepackage{amstext}
\usepackage{amsthm}
\usepackage{amssymb}
\usepackage{setspace}

\makeatletter
\numberwithin{equation}{section}
\theoremstyle{plain}
\newtheorem{thm}{\protect\theoremname}[section]
\theoremstyle{plain}
\newtheorem{prop}[thm]{\protect\propositionname}
\theoremstyle{plain}
\newtheorem{lem}[thm]{\protect\lemmaname}
\theoremstyle{plain}
\newtheorem{cor}[thm]{\protect\corollaryname}
\theoremstyle{definition}
\newtheorem{defn}[thm]{\protect\definitionname}
\theoremstyle{remark}
\newtheorem*{rem*}{\protect\remarkname}

\makeatother

\usepackage{babel}
\providecommand{\corollaryname}{Corollary}
\providecommand{\definitionname}{Definition}
\providecommand{\lemmaname}{Lemma}
\providecommand{\propositionname}{Proposition}
\providecommand{\remarkname}{Remark}
\providecommand{\theoremname}{Theorem}

\begin{document}
\title[Hankel Determinants of Bernoulli Polynomials]{Hankel Determinants of Certain Sequences Of Bernoulli Polynomials:
A Direct Proof of an Inverse Matrix Entry from Statistics}
\author{Lin Jiu}
\address{Zu Chongzhi Center for Mathematics and Computational Sciences, Duke
Kunshan University, Kunshan, Suzhou, Jiangsu Province, PR China, 315216}
\email{lin.jiu@dukekunshan.edu.cn}

\author{Ye Li*}
\address{Class of 2023, Duke
Kunshan University, Kunshan, Suzhou, Jiangsu Province, PR China, 215316.}
\email{ye.li619@dukekunshan.edu.cn}
\thanks{*corresponding author}
\begin{abstract}
We calculate the Hankel determinants of sequences of Bernoulli polynomials.
This corresponding Hankel matrix comes from statistically estimating
the variance in nonparametric regression. Besides its entries' natural
and deep connection with Bernoulli polynomials, a special case of
the matrix can be constructed from a corresponding Vandermonde matrix.
As a result, instead of asymptotic analysis, we give a direct proof
of calculating an entry of its inverse. Further extensions also include
an identity of Stirling numbers of the both kinds. 
\end{abstract}

\maketitle

\section{Introduction}

The \emph{Hankel determinant} of a sequence $\mathbf{c}=(c_{0},c_{1},\ldots)$,
denoted by $H_{n}(\mathbf{c})$ or $H_{n}(c_{k})$, is defined as
the determinant of the \emph{Hankel matrix}, or \emph{persymmetric
matrix}, given by 
\begin{equation}
(c_{i+j})_{0\leq i,j\leq n}=\begin{pmatrix}c_{0} & c_{1} & c_{2} & \cdots & c_{n}\\
c_{1} & c_{2} & c_{3} & \cdots & c_{n+1}\\
c_{1} & c_{2} & c_{3} & \cdots & c_{n+2}\\
\vdots & \vdots & \vdots & \ddots & \vdots\\
c_{n} & c_{n+1} & c_{n+2} & \cdots & c_{2n}
\end{pmatrix}.\label{eq:DEFHankelDet}
\end{equation}
Hankel determinants of various classes of sequences have been extensively
studied, partly due to their close relationship with classical orthogonal
polynomials; see, e.g., \cite[Ch.~2]{Is}; And for numerous results
see, e.g., the very extensive treatments in \cite{Kr1,Kr2,Mi}, and
the numerous references provided there. 

Applications and appearance of Hankel matrices and determinants also
involve statistics. Dai et al.~\cite{Wenlin} studied the following
Hankel determinant from nonparametric regression. Define 
\[
I_{k}:=\sum_{c=1}^{r}c^{k}
\]
and the Hankel matrix 
\[
V_{n}:=\begin{pmatrix}I_{0} & I_{2} & \cdots & I_{2n}\\
I_{2} & I_{4} & \cdots & I_{2n+2}\\
\vdots & \vdots & \ddots & \vdots\\
I_{2n} & I_{2n+2} & \cdots & I_{4n}
\end{pmatrix}.
\]
We shall show when $V_{k}$ is invertible as follows. 
\begin{prop}
\label{prop:invertible}$V_{n}$ is invertible iff $n<r$. 
\end{prop}

In fact, a more specific result gives the left-upper entry of the
inverse is given, for $n=r-1$. One can find a proof in \cite{Inverse},
by asymptotic analysis, for the following result. 
\begin{prop}
\label{prop:Inverse} $(V_{r-1}^{-1})_{1,1}=2\left(\binom{4r}{2r}/\binom{2r}{r}^{2}-1\right)$,
where the left-hand sides is the $(1,1)$ entry of the inverse matrix
of $V_{r-1}$. 
\end{prop}

It is the original purpose of this paper to give a direct computational
proof of Prop.~\ref{prop:Inverse}. Thanks to Dr.~Christian Krattenthaler
for his email, it turns out both Props.~\ref{prop:invertible} and
\ref{prop:Inverse} are direct corollaries of the following expressions. 
\begin{thm}
Let $(a)_{n}=a(a+1)\cdots(a+n-1)$ be the Pochhammer symbol and $B_{n}(x)$
be the $n$th Bernoulli polynomial. 
\begin{equation}
\begin{aligned}H_{n}\left(\frac{B_{2k+5}\left(\frac{x+1}{2}\right)}{2k+5}\right)= & \frac{1}{5\cdot2^{n+2}}\prod_{i=1}^{n}\frac{(2i+3)!^{2}(2i+2)!^{2}}{(4i+5)!(4i+4)!}\prod_{\ell=0}^{n}(x-2n-1+2\ell)_{4n-4\ell+3}\\
 & \times\sum_{i=1}^{n+2}\frac{(2i-1)\left(n+\frac{5}{2}\right)_{i-1}\left(\frac{x}{2}+\frac{1}{2}\right)_{n+2}\left(\frac{x}{2}-n-\frac{3}{2}\right)_{n+2}}{\left(n-i+\frac{5}{2}\right)_{i}(n+2-i)!(n+1+i)!\left(x^{2}-(2i-1)^{2}\right)},
\end{aligned}
\label{eq:Hn2k+5}
\end{equation}
and
\begin{align}
\det V_{n} & =2^{2n^{2}-2n-1}\prod_{i=1}^{n}\frac{(2i)!^{4}}{(4i)!(4i+1)!}\prod_{\ell=0}^{n}(r-\ell)_{2\ell+1}\prod_{\ell=0}^{n-1}\left(r+\frac{1}{2}-\ell\right)_{2\ell+1}\nonumber \\
 & \quad\times\sum_{i=1}^{n+1}\frac{(2n+2i)!(2n+2-2i)!(r+1)_{n+1}}{(n+i)!^{2}(n+1-i)!^{2}(r+i)}.\label{eq:detVn}
\end{align}
\end{thm}

We shall provide a different proof of (\ref{eq:Hn2k+5}), which directly
leads to (\ref{eq:detVn}), aside from the techniques in \cite{FK1}
and \cite{FK2}. More precisely, besides (\ref{eq:Hn2k+5}), we shall
also prove the following Hankel determinants. 
\begin{prop}
\begin{onehalfspace}
\noindent 
\begin{equation}
H_{n}\left(\frac{B_{2k+1}\left(\frac{x+1}{2}\right)}{2k+1}\right)=\left(\frac{x}{2}\right)^{n+1}\prod_{\ell=1}^{n}\left(\frac{(2\ell)^{2}(2\ell-1)^{2}(x^{2}-(2\ell-1)^{2})(x^{2}-(2\ell)^{2})}{16(4\ell-3)(4\ell-1)^{2}(4\ell+1)}\right)^{n+1-\ell}.\label{eq:B2k+1}
\end{equation}

\noindent 
\begin{align}
H_{n}\left(\frac{B_{2k+1}\left(\frac{x+1}{2}\right)}{2k+1}\right) & =\left(\frac{x^{3}-x}{24}\right)^{n+1}\label{eq:B2k+3}\\
 & \quad\times\prod_{\ell=1}^{n}\left(\frac{(2\ell)^{2}(2\ell+1)^{2}(x^{2}-(2\ell+1)^{2})(x^{2}-(2\ell)^{2})}{16(4\ell-1)(4\ell+1)^{2}(4\ell+3)}\right)^{n+1-\ell}.\nonumber 
\end{align}
\end{onehalfspace}
\end{prop}

Note that these three sequences are different from the recent work
\cite{DJ1,DJ2,DJ3}, by Dilcher and the first author, on the Hankel
determinants of sequences related to Bernoulli and Euler polynomial.
Although the Hankel determinant of $H_{n}\left(B_{2k+1}(\tfrac{x+1}{2})\right)$
is obtained in \cite[Thm.~1.1]{DJ1}, the expression of $H_{n}\left(\frac{B_{2k+1}\left(\frac{x+1}{2}\right)}{2k+1}\right)$
is quite different, which, in general, is true for $H_{n}(a_{k})$,
$H_{n}(a_{k}/k)$, and $H_{n}(ka_{k-1})$. Take the the Euler numbers
$E_{k}$ as an example: we have \cite[Eq.~(4.2)]{Euler}
\[
H_{n}(E_{k})=(-1)^{\binom{n+1}{2}}\prod_{\ell=1}^{n}\ell!^{2},
\]
and \cite[Cor.~3.4]{DJ2},
\[
H_{n}\left(kE_{k-1}\right)=\begin{cases}
0, & k=2m;\\
(-1)^{m+1}2^{4m(m+1)}\overset{m}{\underset{\ell=1}{\prod}}\ell!^{8}, & k=2m+1.
\end{cases}
\]
Namely, the latter also depends on the parity of the dimension. 

This paper is structured as follows. In Section \ref{sec:Preliminaries},
we first quote important results in Bernoulli polynomials, orthogonal
polynomials, continued fractions, and polygamma functions, required
in later sections. In Section \ref{sec:HD}, we give the proofs of
three Hankel determinants (\ref{eq:B2k+1}), (\ref{eq:B2k+3}), and
(\ref{eq:Hn2k+5}). In Section \ref{sec:Vk}, besides the proof of
(\ref{eq:detVn}) some further results on $V_{k}$ are given, including
alternative proofs of Props.~\ref{prop:invertible} and \ref{prop:Inverse},
rather than direct corollaries from (\ref{eq:detVn}). This approach
leads an identity involving Stirling numbers, stated and proven finally
in Section \ref{sec:FinalRemarks}, together with some further remarks. 

\section{\label{sec:Preliminaries}Preliminaries}

All the necessary background stated here in this section can be found
in \cite{DJ1,DJ2,DJ3}, in concise form. We repeat this material here
for easy reference, and to make this paper self-contained.

\subsection{Bernoulli numbers, Bernoulli polynomials and connection with $V_{n}$}

The Bernoulli numbers $B_{n}$ and Bernoulli polynomials $B_{n}(x)$
are usually defined by their exponential generating functions 

\[
\frac{te^{tx}}{e^{t}-1}=\sum_{n=0}^{\infty}B_{n}(x)\frac{t^{n}}{n!}\quad\text{and}\quad\frac{te^{tx}}{e^{t}-1}=\sum_{n=0}^{\infty}B_{n}(x)\frac{t^{n}}{n!}.
\]
Their application and connections include number theory, combinatorics,
numerical analysis, and other areas. The direct connection between
$B_{n}(x)$ and $I_{k}$ can be derived from 
\[
B_{n+1}(x+1)-B_{n+1}(x)=(n+1)x^{n};
\]
(see e.g., \cite[Entry (24.4.1)]{NIST}), so that
\[
I_{k}=\frac{B_{k+1}(r+1)-B_{k+1}(1)}{k+1}.
\]
Note that $B_{2k+1}(1)=0$ (see e.g., \cite[Entries (24.2.4), (24.4.3)]{NIST})
for all positive integers $k$. Hence, 
\begin{equation}
I_{2k}=\begin{cases}
\frac{B_{2k+1}(r+1)}{2k+1}, & k>0;\\
B_{1}(r+1)-B_{1}(1)=r, & k=0.
\end{cases}\label{eq:IandB}
\end{equation}
Namely, $\det V_{n}$ is almost the same as $H_{n-1}\left(\frac{B_{2k+1}(r+1)}{2k+1}\right)$,
(please note the difference of the dimensions) except for $I_{0}=r$
and $B_{1}(r+1)=r+1/2$. 

\subsection{Orthogonal polynomials}

Suppose we are given a sequence ${\bf c}=(c_{0},c_{1},\ldots)$ of
numbers; then we can define a linear functional $L$ on polynomials
by 
\begin{equation}
L(x^{k})=c_{k},\quad k=0,1,2,\ldots.\label{eq:Operator}
\end{equation}
We may also normalize the sequence such that $c_{0}=1$. We now summarize
several well-known facts and state them as a lemma with two corollaries;
see, e.g., \cite[Ch.~2]{Is} and \cite[pp.~7--10]{Chi}.
\begin{lem}
\label{lem:Orthogonality1} Let $L$ be the linear functional in (\ref{eq:Operator}).
If (and only if) $H_{n}(c_{k})\neq0$ for all $n=0,1,2,\ldots$, there
exists a unique sequence of monic polynomials $P_{n}(y)$ of degree
$n$, $n=0,1,\ldots$, and a sequence of positive numbers $(\zeta_{n})_{n\geq1}$,
with $\zeta_{0}=1$, such that 
\begin{equation}
L\left(P_{m}(y)P_{n}(y)\right)=\zeta_{n}\delta_{m,n},\label{eq:Orthogonality1}
\end{equation}
where $\delta_{m,n}$ is the Kronecker delta function. Furthermore,
for all $n\geq1$ we have $\zeta_{n}=H_{n}({\bf c})/H_{n-1}({\bf c})$,
and for $n\geq1$, 
\begin{equation}
P_{n}(y)=\frac{1}{H_{n-1}({\bf c})}\det\begin{pmatrix}c_{0} & c_{1} & \cdots & c_{n}\\
c_{1} & c_{2} & \cdots & c_{n+1}\\
\vdots & \vdots & \ddots & \vdots\\
c_{n-1} & c_{n} & \cdots & c_{2n-1}\\
1 & y & \cdots & y^{n}
\end{pmatrix},\label{eq:OrthPolyDet}
\end{equation}
where the polynomials $P_{n}(y)$ satisfy the $3$-term recurrence
relation $P_{0}(y)=1$, $P_{1}(y)=y+s_{0}$, and 
\begin{equation}
P_{n+1}(y)=(y+s_{n})P_{n}(y)-t_{n}P_{n-1}(y)\qquad(n\geq1),\label{eq:3TermRec}
\end{equation}
for some sequences $(s_{n})_{n\geq0}$ and $(t_{n})_{n\geq1}$.
\end{lem}

We now multiply both sides of (\ref{eq:OrthPolyDet}) by $y^{r}$
and replace $y^{j}$ by $c_{j}$, which includes replacing the constant
term $1$ by $c_{0}$ for $r=0$. Then for $0\leq r\leq n-1$ the
last row of the matrix in (\ref{eq:OrthPolyDet}) is identical with
one of the previous rows, and thus the determinant is $0$. When $r=n$,
the determinant is $H_{n}({\bf c})$. We therefore have the following
result.
\begin{cor}
\label{cor:Orthgonality2} With the sequence $(c_{k})$ and the polynomials
$P_{n}(y)$ as above, we have 
\begin{equation}
y^{r}P_{n}(y)\bigg|_{y^{k}=c_{k}}=\begin{cases}
0, & 0\leq r\leq n-1;\\
H_{n}({\bf c})/H_{n-1}({\bf c}), & r=n.
\end{cases}\label{eq:Orthgonality2}
\end{equation}
\end{cor}

The polynomials $P_{n}(y)$ are known as ``the monic orthogonal polynomials
belonging to the sequence ${\bf c}=(c_{0},c_{1},\ldots)$'', or ``the
polynomials orthogonal with respect to ${\bf c}$''. 

The next result which we require establishes a connection with certain
continued fractions (in this case called \emph{J-fractions}). It can
be found in various relevant publications, for instance in \cite[p.~20]{Kr1}. 
\begin{lem}
\label{lem:HankelDetCF}Let ${\bf c}=(c_{k})_{k\geq0}$ be a sequence
of numbers with $c_{0}\neq0$, and suppose that its generating function
is written in the form 
\[
\sum_{k=0}^{\infty}c_{k}z^{k}=\frac{c_{0}}{1+s_{0}z-\frac{t_{1}z^{2}}{1+s_{1}z-\frac{t_{2}z^{2}}{1+s_{2}z-\ddots}}}
\]
 where both sides are considered as formal power series. Then the
sequences $(s_{n})$ and $(t_{n})$ are the same as in (\ref{eq:3TermRec});
and we have

\begin{equation}
H_{n}({\bf c})=c_{0}^{n+1}t_{1}^{n}t_{2}^{n-1}\cdots t_{n-1}^{2}t_{n}\qquad(n\geq0).\label{eq:HankelDetTn}
\end{equation}
\end{lem}

Following \cite{MWY}, we consider the infinite band matrix 
\begin{equation}
J:=\begin{pmatrix}-s_{0} & 1 & 0 & 0 & \cdots\\
t_{1} & -s_{1} & 1 & 0 & \cdots\\
0 & t_{2} & -s_{2} & 1 & \cdots\\
\vdots & \vdots & \ddots & \ddots & \ddots
\end{pmatrix}.\label{eq:J}
\end{equation}
Furthermore, for each $n\geq0$ let $J_{n}$ be the $(n+1)$th leading
principal submatrix of $J$ and let 
\begin{equation}
D_{n}:=\det J_{n},\label{eq:Dn}
\end{equation}
so that $D_{0}=-s_{0}$. We also set $D_{-1}=1$ by convention, and
furthermore, using elementary determinant operations, we get from
(\ref{eq:J}) the recurrence relation 
\begin{equation}
D_{n+1}=-s_{n+1}D_{n}-t_{n+1}D_{n-1}.\label{eq:RecurrrenceDn}
\end{equation}
We can now quote the following results.
\begin{lem}
\label{lem:LeftShifted} \cite[Prop.~1.2]{MWY} With notation as above,
for a given sequence ${\bf c}$ we have 
\begin{equation}
H_{n}(c_{k+1})=H_{n}(c_{k})\cdot D_{n},\label{eq:HankelDetLeftShifted1}
\end{equation}
and 
\begin{equation}
H_{n}(c_{k+2})=H_{n}(c_{k})\cdot\left(\prod_{\ell=1}^{n+1}t_{\ell}\right)\cdot\sum_{\ell=-1}^{n}\frac{D_{\ell}^{2}}{\prod_{j=1}^{\ell+1}t_{j}}.\label{eq:eq:HankelDetLeftShifted2}
\end{equation}
\end{lem}

\begin{lem}
\cite[Eq.~(2.4)]{Han}\label{lem:LeftShiftedSn} For a given sequence
${\bf c}$ and $(s_{n})$ as defined above, we have 
\begin{equation}
s_{n}=-\frac{1}{H_{n-1}(c_{k+1})}\left(\frac{H_{n-1}(c_{k})H_{n}(c_{k+1})}{H_{n}(c_{k})}+\frac{H_{n}(c_{k})H_{n-2}(c_{k+1})}{H_{n-1}(c_{k})}\right).\label{eq:LeftShiftedSn}
\end{equation}
\end{lem}

\subsection{Continued fractions}

Following the usage in books such as \cite{CF} or \cite{CF2}, we
write 
\begin{equation}
b_{0}+\overset{\infty}{\text{\ensuremath{\underset{m=1}{\mathbf{K}}}}}(a_{m}/b_{m})=b_{0}+\mathbf{K}(a_{m}/b_{m})=b_{0}+\frac{a_{1}}{b_{1}+\frac{a_{2}}{b_{2}+\ddots}}\label{eq:InfiniteCF}
\end{equation}
for an \emph{infinite continued fraction}. The $n$th \emph{approximant}
is expressed by 
\begin{equation}
b_{0}+\overset{n}{\text{\ensuremath{\underset{m=1}{\mathbf{K}}}}}(a_{m}/b_{m})=b_{0}+\frac{a_{1}}{b_{1}+\ddots+\frac{a_{n}}{b_{n}}}=\frac{A_{n}}{B_{n}},\label{eq:nApproximant}
\end{equation}
and $A_{n}$, $B_{n}$ are called the $n$th \emph{numerator} and
\emph{denominator}, respectively. The continued fraction (\ref{eq:InfiniteCF})
is said to \emph{converge} if the sequence of approximants in (\ref{eq:nApproximant})
converges. In this case, the limit is called the \emph{value} of the
continued fraction (\ref{eq:InfiniteCF}). 

Two continued fractions are said to be \emph{equivalent} if and only
if they have the same sequences of approximants. In other words, we
have 
\[
b_{0}+\overset{n}{\text{\ensuremath{\underset{m=1}{\mathbf{K}}}}}(a_{m}/b_{m})=d_{0}+\overset{n}{\text{\ensuremath{\underset{m=1}{\mathbf{K}}}}}(c_{m}/d_{m})
\]
if and only if there exists a sequence of nonzero complex numbers
$(r_{m})_{m\geq0}$ with $r_{0}=1$, such that for $m\geq0$, 
\begin{equation}
d_{m}=r_{m}b_{m}\quad\text{and}\quad c_{m+1}=r_{m+1}r_{m}a_{m+1};\label{eq:EquivalentCFs}
\end{equation}
(see \cite[Eq.~(1.4.2)]{CF}). We also require the following special
case of the more general concept of a contraction; see, e.g., \cite[p.~16]{CF}.
\begin{defn}
Let $A_{n}$, $B_{n}$ be the $n$th numerator and denominator, respectively,
of a continued fraction $\mathrm{cf}_{1}:=b_{0}+\mathbf{K}(a_{m}/b_{m})$
, and let $C_{n}$, $D_{n}$ be the corresponding quantities of $\mathrm{cf}_{2}:=d_{0}+\mathbf{K}(c_{m}/d_{m})$.
Then $\mathrm{cf}_{2}$ is called an \emph{even canonical contraction}
of $\mathrm{cf}_{1}$ if 
\[
C_{n}=A_{2n}\qquad D_{n}=B_{2n}\qquad(n\geq0);
\]
and is called an \emph{odd canonical contraction} of cf1 if 
\[
C_{0}=\frac{A_{1}}{B_{1}},\quad D_{0}=1,\quad C_{n}=A_{2n+1},\quad D_{n}=B_{2n+1}\qquad(n\geq0).
\]
\end{defn}

We will now state three identities that will be used in later sections;
see \cite[pp.~16--18]{CF}, \cite[pp.~83--85]{CF2}, or \cite[pp.~21--21]{CFWall}
for proofs and further details. 
\begin{lem}
An even canonical contraction of $b_{0}+\mathbf{K}(a_{m}/b_{m})$
exists if and only if $b_{2k}\neq0$ for $k\geq1$, and we have 
\[
b_{0}+\overset{\infty}{\underset{m=1}{\mathbf{K}}}(a_{m}/b_{m})=b_{0}+\frac{a_{1}b_{2}}{a_{2}+b_{1}b_{2}-\frac{a_{2}a_{3}\frac{b_{4}}{b_{2}}}{a_{4}+b_{3}b_{4}+a_{3}\frac{b_{4}}{b_{2}}-\frac{a_{4}a_{5}\frac{b_{6}}{b_{4}}}{a_{6}+b_{5}b_{6}+a_{5}\frac{b_{6}}{b_{4}-\frac{a_{6}a_{7}\frac{b_{8}}{b_{6}}}{\ddots}}}}}
\]
In particular, with $b_{0}=0$, $b_{k}=1$ for $k\geq1$, $a_{1}=1$,
and $a_{k}=\alpha_{k-1}t$ ($k\geq1$), for some variable $t$, we
have 
\begin{equation}
\frac{1}{1-\frac{\alpha_{1}t}{1-\frac{\alpha_{2}t}{1-\ddots}}}=\frac{1}{1-\alpha_{1}t-\frac{\alpha_{1}\alpha_{2}t^{2}}{1-(\alpha_{2}+\alpha_{3})t-\frac{\alpha_{3}\alpha_{4}t^{2}}{1-(\alpha_{4}+\alpha_{5})t-\frac{\alpha_{5}\alpha_{6}t^{2}}{\ddots}}}}\label{eq:EvenContraction}
\end{equation}
Similarly, an odd canonical contraction gives 
\begin{equation}
1+\frac{\alpha_{1}t}{1-(\alpha_{1}+\alpha_{2})t-\frac{\alpha_{2}\alpha_{3}t^{2}}{1-(\alpha_{3}+\alpha_{4})t-\frac{\alpha_{4}\alpha_{5}t^{2}}{1-(\alpha_{5}+\alpha_{6})t-\frac{\alpha_{6}\alpha_{7}t^{2}}{\ddots}}}},\label{eq:OddContraction}
\end{equation}
 for the continued fraction on the left-hand side of (\ref{eq:EvenContraction}).
\end{lem}

\subsection{Polygamma functions}

We shall use the \emph{polygamma function of order $m$}: $\psi^{(m)}(z):=\frac{d^{m+1}}{dz^{m+1}}\log\Gamma(z)$,
with 
\[
\psi(z)=\psi^{(0)}(z)=\frac{d}{dz}\left(\log\Gamma(z)\right)=\frac{\Gamma'(z)}{\Gamma(z)}.
\]
Among all the properties, we first recall the following well-known
complete asymptotic expansion, valid for $|\arg z|<\pi$: (see e.g.,
\cite[p.~48, Eq.~(12)]{ErMaObTr})
\begin{equation}
\log\Gamma(z+x)=\left(z+x-\frac{1}{2}\right)\log z-z+\frac{\log(2\pi)}{2}+\sum_{n=1}^{\infty}\frac{(-1)^{n+1}B_{n+1}(x)}{n(n+1)z^{n}},\label{eq:AsymptoticLogGamma}
\end{equation}
which, by differentiating both sides with respect to $z$, yields
the asymptotic expansion of $\psi(z)$:

\begin{onehalfspace}
\noindent 
\begin{equation}
\psi(z+x)=\log z+\frac{x-\frac{1}{2}}{z}-\sum_{n=1}^{\infty}\frac{(-1)^{n+1}B_{n+1}(x)}{(n+1)z^{n+1}}.\label{eq:AsymptoticPsi}
\end{equation}
Meanwhile, the series expansion formula (see e.g., \cite[Entry 5.~7.~6]{NIST})

\noindent 
\begin{equation}
\psi(z)=-\gamma-\frac{1}{z}+\sum_{k=1}^{\infty}\frac{z}{k(k+z)}=-\gamma+\sum_{k=0}^{\infty}\left(\frac{1}{k+1}-\frac{1}{k+z}\right),\label{eq:PsiSeries}
\end{equation}
later shall lead us to the continued fraction expressions that are
crucial to our proofs. 
\end{onehalfspace}

\section{\label{sec:HD}Hankel Determinants}

We begin with our three main Hankel determinants. 

\subsection{$B_{2k+1}\left(\frac{x+1}{2}\right)/(2k+1)$}
\begin{proof}
[Proof of \eqref{eq:B2k+1}] From (\ref{eq:AsymptoticPsi}), we have
\begin{align*}
 & \quad\psi\left(z+\frac{1+x}{2}\right)-\psi\left(z+\frac{1-x}{2}\right)\\
 & =\frac{x}{z}+\sum_{n=1}^{\infty}\frac{(-1)^{n+1}}{(n+1)z^{n+1}}\left(B_{n+1}\left(\tfrac{1-x}{2}\right)-B_{n+1}\left(\tfrac{1+x}{2}\right)\right)=2\sum_{n=0}^{\infty}\frac{B_{2n+1}\left(\frac{1+x}{2}\right)}{(2n+1)z^{2n+1}},
\end{align*}
where we used the reflection formula (see e.g., \cite[Entry  24.~4.~23]{NIST})
so that 

\begin{onehalfspace}
\noindent 
\[
B_{n+1}\left(\frac{1-x}{2}\right)=B_{n+1}\left(1-\frac{1+x}{2}\right)=(-1)^{n}B_{n+1}\left(\frac{1+x}{2}\right)
\]
and $B_{1}(x)=x-1/2$. This, finally by the change of variables $z\mapsto1/z$,
implies

\noindent 
\[
\sum_{n=0}^{\infty}\frac{B_{2n+1}(\frac{1+r}{2})}{2n+1}z^{2n}=\frac{\psi\left(\frac{1}{z}+\frac{1+x}{2}\right)-\psi\left(\frac{1}{z}+\frac{1-x}{2}\right)}{2z}.
\]
We denote the left-hand side of the above equation by $F(z),$ where
the independence on $x$ is implied. By the series expansion (\ref{eq:PsiSeries}),
we see

\noindent 
\[
zF(z)=\sum_{k=0}^{\infty}\left(\frac{1}{\frac{2}{z}-x+2k+1}-\frac{1}{\frac{2}{z}+x+2k+1}\right).
\]

\noindent Now we use the following continued fraction due to Ramanujan
(see e.g., \cite[p.~149, Entry 30]{RamanujanII}) for either $n$
is an integer or $Re(t)>0$: 

\noindent 
\[
\sum_{k=0}^{\infty}\left(\frac{1}{t-n+2k+1}-\frac{1}{t+n+2k+1}\right)=\frac{n}{t+\frac{1(1-n^{2})}{3t+\frac{2^{2}(2^{2}-n^{2})}{5t+\frac{3^{2}(3^{2}-n^{2})}{7t+\ddots}}}},
\]
by letting $t=2/z$ and $n=x$, to get

\noindent 
\[
zF(z)=\frac{x}{\frac{2}{z}+\frac{1(1-x^{2})}{\frac{6}{z}+\frac{2^{2}(2^{2}-x^{2})}{\frac{10}{z}+\frac{3^{2}(3^{2}-x^{2})}{\frac{14}{z}+\ddots}}}}.
\]
Using equivalence of continued fractions (\ref{eq:EquivalentCFs}),
with
\[
r_{m}=\begin{cases}
1, & m=0;\\
\frac{z}{2(2m-1)}, & m\geq1,
\end{cases}\quad a_{m}=\begin{cases}
x, & m=1;\\
(m-1)^{2}((m-1)^{2}-x^{2}), & m\geq2,
\end{cases}
\]
and $b_{m}=\frac{2(2m-1)}{z}$, we could get

\noindent 
\begin{align*}
d_{m} & =b_{m}r_{m}=\frac{2(2m-1)}{z}\frac{z}{2(2m-1)}=1,\\
c_{m+1} & =r_{m+1}r_{m}a_{m+1}=\frac{m^{2}(m^{2}-r^{2})z^{2}}{4(2m+1)(2m-1)}.
\end{align*}
Hence, after dividing both sides by $xz/2$, we have

\noindent 
\begin{equation}
\frac{F(z)}{\frac{x}{2}}=\text{\ensuremath{\frac{1}{1+\frac{\frac{(1-x^{2})z^{2}}{12}}{1+\frac{\frac{4(4-x^{2})z^{2}}{15}}{1+\ddots}}}}. }\label{eq:CFFz}
\end{equation}
For simplification, we define
\begin{equation}
\text{\ensuremath{\alpha_{m}=\frac{m^{2}(x^{2}-m^{2})}{4(2m+1)(2m-1)}}},\label{eq:AlphaM}
\end{equation}
and apply the even canonical contraction (\ref{eq:EvenContraction})
on (\ref{eq:CFFz}), to obtain

\noindent 
\begin{align*}
\tau_{m}^{(0)}=\alpha_{2m-1}\alpha_{2m} & =\frac{(2m-1)^{2}(2m)^{2}(x^{2}-(2m-1)^{2})(x^{2}-(2m)^{2})}{16(4m-3)(4m-1)^{2}(4m+1)},\\
\sigma_{m}^{(0)}=\alpha_{2m}+\alpha_{2m+1} & =\frac{(8m^{2}+4m-1)r^{2}-(32m^{4}+32m^{3}+8m^{2}-1)}{4(4m+3)(4m+1)(4m-1)},
\end{align*}
such that 
\[
F(z)=\frac{\frac{x}{2}}{1+\sigma_{0}^{(0)}z^{2}-\frac{\tau_{1}^{(0)}z^{4}}{1+\sigma_{1}^{(0)}z^{2}-\frac{\tau_{2}^{(0)}z^{4}}{1+\sigma_{1}^{(0)}z^{2}-\ddots}}}.
\]
Then use Lem.~\ref{lem:HankelDetCF} with $c_{0}=\frac{x}{2}$, $s_{j}=-\sigma_{j}^{(0)}$,
and $t_{j}=\tau_{j}^{(0)}$, we have the desired result (\ref{eq:B2k+1}).
\qedhere
\end{onehalfspace}
\end{proof}

\subsection{$B_{2k+3}\left(\frac{x+1}{2}\right)/(2k+3)$}
\begin{proof}
[Proof of \eqref{eq:B2k+3}] Similarly, let 
\[
G(z)=\sum_{n=0}^{\infty}\frac{B_{2n+3}(\frac{1+x}{2})}{2n+3}z^{2n+2}=F(z)-B_{1}\left(\frac{1+x}{2}\right)=F(z)-\frac{x}{2}.
\]
Then, by (\ref{eq:CFFz}), 
\begin{equation}
\frac{G(z)}{\frac{x}{2}}=\frac{F(z)}{\frac{x}{2}}-1=\text{\ensuremath{\frac{1}{1+\frac{\frac{(1-x^{2})z^{2}}{12}}{1+\frac{\frac{4(4-x^{2})z^{2}}{15}}{1+\ddots}}}-1}}.\label{eq:CFGz}
\end{equation}
Focus on the first continued fractions above. Using odd canonical
contraction (\ref{eq:OddContraction}) with $\alpha_{m}$ defined
in (\ref{eq:AlphaM}), we would have
\[
G(z)=\frac{\frac{x^{3}-x}{24}}{1-\sigma_{0}^{(1)}z^{2}-\frac{\tau_{1}^{(1)}z^{4}}{1-\sigma_{1}^{(1)}z^{2}-\frac{\tau_{2}^{(1)}z^{4}}{1-\ddots}}},
\]
where
\begin{align}
\tau{}_{m}^{(1)}=\alpha{}_{2m}\alpha{}_{2m+1} & =\frac{(2m)^{2}(2m+1)^{2}(x^{2}-(2m))(x^{2}-(2m+1)^{2})}{16(4m-1)(4m+1)^{2}(4m+3)}\label{eq:tau1}\\
\sigma_{m}^{(1)}=\alpha{}_{2m+1}+\alpha{}_{2m+2} & =\frac{(2m+1)^{2}(x^{2}-(2m+1)^{2})}{4(4m+3)(4m+1)}+\frac{(2m+2)^{2}(x^{2}-(2m+2)^{2})}{4(4m+5)(4m+3)}.\label{eq:sigma1}
\end{align}

\begin{onehalfspace}
\noindent Then use Lem.~\ref{lem:HankelDetCF} with $c_{0}=\frac{x^{3}-x}{24}=\frac{x}{2}\alpha_{1}$,
$s_{j}=-\sigma_{j}^{(1)}$, and $t_{j}=\tau_{j}^{(1)}$, we have the
desired result. 
\end{onehalfspace}
\end{proof}
\begin{rem*}
Noting that $c_{k}=b_{k+1}$, so the ``left-shifted'' sequence formula
applies here and leads to the same result (\ref{eq:B2k+3}). In fact,
we can easily derive that 
\begin{align}
D_{n}^{(0)} & :=\det\begin{pmatrix}\sigma_{0}^{(0)} & 1 & 0 & 0 & \cdots & 0\\
\tau_{1}^{(0)} & \sigma_{1}^{(0)} & 1 & 0 & \cdots & 0\\
0 & \tau_{2}^{(0)} & \sigma_{2}^{(0)} & 1 & \cdots & 0\\
\vdots & \vdots & \ddots & \ddots & \ddots & \vdots\\
0 & \cdots & 0 & 0 & \tau_{n}^{(0)} & \sigma_{n}^{(0)}
\end{pmatrix}\label{eq:Dn0}\\
 & =\frac{1}{4^{n+1}}\prod_{\ell=1}^{n+1}\frac{(x^{2}-(2\ell+1)^{2})(2\ell-1)^{2}}{(4\ell-1)(4\ell-3)},\nonumber 
\end{align}
so that 
\begin{equation}
H_{n}\left(\frac{B_{2k+3}\left(\frac{x+1}{2}\right)}{2k+3}\right)=H_{n}\left(\frac{B_{2k+1}\left(\frac{x+1}{2}\right)}{2k+1}\right)D_{n}^{(0)}.\label{eq:bkck}
\end{equation}
\end{rem*}

\subsection{$B_{2k+5}\left(\frac{x+1}{2}\right)/(2k+5)$}

The following proof contains some tedious calculation and simplification,
the details of which are omitted. 
\begin{proof}
[Proof of \eqref{eq:Hn2k+5}] Similarly as the remark above, it suffices
to show 
\[
\frac{H_{n}\left(\frac{B_{2k+5}\left(\frac{x+1}{2}\right)}{2k+5}\right)}{H_{n}\left(\frac{B_{2k+3}\left(\frac{x+1}{2}\right)}{2k+3}\right)}=:D_{n}^{(1)}=\det\begin{pmatrix}\sigma_{0}^{(1)} & 1 & 0 & 0 & \cdots & 0\\
\tau_{1}^{(1)} & \sigma_{1}^{(1)} & 1 & 0 & \cdots & 0\\
0 & \tau_{2}^{(1)} & \sigma_{2}^{(1)} & 1 & \cdots & 0\\
\vdots & \vdots & \ddots & \ddots & \ddots & \vdots\\
0 & \cdots & 0 & 0 & \tau_{n}^{(1)} & \sigma_{n}^{(1)}
\end{pmatrix},
\]
or equivalently, by (\ref{eq:RecurrrenceDn}), to show 
\begin{equation}
D_{n+1}^{(1)}=\sigma_{n+1}^{(1)}D_{n}^{(1)}-\tau_{n+1}^{(1)}D_{n-1}^{(1)},\label{eq:RecD1}
\end{equation}
where $H_{n}\left(B_{2k+5}\left(\frac{x+1}{2}\right)/(2k+5)\right)$
is given by (\ref{eq:Hn2k+5}), and $\tau_{n}^{(1)}$ and $\sigma_{n}^{(1)}$
are given by (\ref{eq:tau1}) and (\ref{eq:sigma1}), respectively. 
\begin{enumerate}
\item First of all, by simplification, we see 
\begin{align*}
\left(\frac{x}{2}+\frac{1}{2}\right)_{n+2}\left(\frac{x}{2}-n-\frac{3}{2}\right)_{n+2} & =\frac{1}{2^{2n+4}}\prod_{i=1}^{n+2}(x-(2i-1)^{2}).
\end{align*}
\item In addition, let 
\[
f(n,x):=\frac{1}{5\cdot2^{n+2}}\prod_{i=1}^{n}\frac{(2i+3)!^{2}(2i+2)!^{2}}{(4i+5)!(4i+4)!}\prod_{l=0}^{n}(x-2n-1+2l)_{4n-4l+3},
\]
and the recurrence, upon simplification, 
\[
\frac{f(n+1,x)}{f(n,x)}=\frac{x(x^{2}-1)}{2}\cdot\frac{(2n+5)!^{2}(2n+4)!^{2}}{(4n+9)!(4n+8)!}\prod_{l=1}^{n+1}\left[x^{2}-\left(2l+1\right)^{2}\right]\left[x^{2}-\left(2l\right)^{2}\right]
\]
indicates 
\[
\frac{H_{n+1}\left(\frac{B_{2k+3}\left(\frac{x+1}{2}\right)}{2k+3}\right)}{H_{n}\left(\frac{B_{2k+3}\left(\frac{x+1}{2}\right)}{2k+3}\right)}=\frac{f(n+1,x)}{f(n,x)}\cdot\frac{4(4n+9)(4n+7)}{(2n+5)^{2}(2n+4)^{2}}.
\]
\item Due to $(a)_{n}=\Gamma(a+n)/\Gamma(a)$ and $\Gamma\left(\frac{1}{2}+n\right)=(2n-1)!!\sqrt{\pi}/2^{n}$,
we see 
\begin{align*}
\frac{\left(n+\frac{5}{2}\right)_{i-1}}{\left(n-i+\frac{5}{2}\right)_{i}(n+2-i)!(n+1+i)!} & =\frac{\Gamma\left(n+i+\frac{3}{2}\right)\Gamma\left(n-i+\frac{5}{2}\right)}{\Gamma^{2}\left(n+\frac{5}{2}\right)\Gamma(n+3-i)\Gamma(n+2+i)}\\
 & =\frac{4^{n+2}}{(2n+3)!!^{2}\pi}\frac{\Gamma\left(n+i+\frac{3}{2}\right)\Gamma\left(n-i+\frac{5}{2}\right)}{\Gamma(n+3-i)\Gamma(n+2+i)}.
\end{align*}
\end{enumerate}
Therefore, it suffices to show that 
\[
D_{n}^{(1)}=\frac{2(n+1)!^{2}}{(4n+5)!!\pi}\prod_{i=1}^{n+2}(x^{2}-(2i-1)^{2})\sum_{i=1}^{n+2}\frac{(2i-1)\Gamma\left(n+i+\frac{3}{2}\right)\Gamma\left(n-i+\frac{5}{2}\right)}{\Gamma(n+3-i)\Gamma(n+2+i)\left(x^{2}-(2i-1)^{2}\right)},
\]
satisfies (\ref{eq:RecD1}). Although the product term $\prod_{i=1}^{n+2}(x^{2}-(2i-1)^{2})$
has degree $2n+4$, the sum with $x^{2}-(2i-1)^{2}$ in the denominator
will always cancel one factor. Therefore, $D_{n+1}^{(1)}$ is a polynomials
in $x$ of degree at most $2n+2$. Also noting that $\tau_{n}^{(1)}$
is of degree $4$; while $\sigma_{n}^{(1)}$ of degree $2$, in $x$,
we see (\ref{eq:RecD1}) is basically showing two polynomials of degree
at most $2(n+2)$ are identical. Therefore, as long as the left-hand
side matches the right-hand side at $2n+5$ different values of $x$,
(\ref{eq:RecD1}) holds for any $x$, and the proof is complete. 
\begin{enumerate}
\item Consider $x=\pm(2j-1)$, where $j=1,2,\ldots,n+1$ (i.e., $2n+2$
different points). In the case, all the terms in the summation of
$D_{n}^{(1)}$ will vanish, except for exactly the one with $i=2j-1$,
since the factor $x^{2}-(2i-1)^{2}$ is canceled with the previous
product, making it 
\[
\left(\prod_{i=1}^{j-1}((2j-1)^{2}-(2i-1)^{2})\right)\left(\prod_{i=j+1}^{n+2}((2j-1)^{2}-(2i-1)^{2})\right)
\]
Then, by canceling the common product term and after simplification,
we see the left-hand side of (\ref{eq:RecD1}) is given by, 
\[
\left(\prod_{i=n+2}^{n+3}((2j-1)^{2}-(2i-1)^{2})\right)\frac{2(n+2)!^{2}(2j-1)}{4^{2n+5}(4n+9)!!}\binom{2n+4+2j}{n+2+j}\binom{2n+6-2j}{n+3-j},
\]
while the right-hand side is 
\begin{align*}
 & \quad\left(\frac{(2n+4)^{2}((2j-1)^{2}-(2n+4)^{2})}{4(4n+7)(4n+9)}+\frac{(2n+3)^{2}((2j-1)^{2}-(2n+3)^{2})}{4(4n+7)(4n+5)}\right)\\
 & \times\left(((2n-1)^{2}-(2n+3)^{2})\right)\frac{2(n+1)!^{2}(2j-1)}{4^{2n+3}(4n+5)!!}\binom{2n+2+2j}{n+1+j}\binom{2n+4-2j}{n+2-j}\\
 & -\frac{(2n+3)^{2}(2n+2)^{2}((2j-1)^{2}-(2n+3)^{2})((2j-1)^{2}-(2n+2)^{2})}{16(4n+7)(4n+3)(4n+5)^{2}}\\
 & \times\frac{2(n!)^{2}(2j-1)}{4^{2n+1}(4n+1)!!}\binom{2n+2j}{n+j}\binom{2n+2-2j}{n+1-j}.
\end{align*}
Further simplification shows that (\ref{eq:RecD1}) with $x=\pm(2j-1)$,
for $j=1,2,\ldots,n+1$ is equivalent to 
\begin{align*}
 & \quad\left((2j-1)^{2}-(2n+5)^{2}\right)\frac{(n+2)^{2}}{4(4n+9)}\frac{(2n+3+2j)(2n+5-2j)}{(n+2+j)(n+3-j)}\\
 & =\left(\frac{(2n+4)^{2}((2j-1)^{2}-(2n+4)^{2})}{4(4n+9)}+\frac{(2n+3)^{2}((2j-1)^{2}-(2n+3)^{2})}{4(4n+5)}\right)\\
 & \quad+\frac{(2n+3)^{2}(n+1+j)(n+2-j)}{(4n+5)},
\end{align*}
which is trivial to verify. 
\item Let $x=\pm(2n+3)$, i.e., $x=\pm(2j-1)$ for $j=n+2$. Note that in
this case, the summation in $D_{n-1}^{(1)}$ will not reduce to a
single term; but $\tau_{n+1}^{(1)}=0$, which simplifies (\ref{eq:RecD1})
into 
\begin{align*}
 & \quad\left((2n+3)^{2}-(2n+5)^{2}\right)\frac{2(n+2)!^{2}(2n+3)}{4^{2n+5}(4n+9)!!}\binom{4n+8}{2n+4}\binom{2}{1}\\
 & =\frac{(2n+4)^{2}((2n+3)^{2}-(2n+4)^{2})}{4(4n+7)(4n+9)}\frac{2(n+1)!^{2}2n+3}{4^{2n+3}(4n+5)!!}\binom{4n+6}{2n+3}\binom{0}{0},
\end{align*}
which is equivalent to the trivial identity 
\[
(n+2)\binom{4n+8}{2n+4}=(4n+7)\binom{4n+6}{2n+3}.
\]
\item Now, we have already checked $2n+4$ different points, so one more
is adequate. Note that $x=2n+5$ does not work, since in this case,
the summation in both $D_{n}^{(1)}$ and $D_{n-1}^{(1)}$ remain.
Instead, we take $x=2n+2$, so that 
\[
\sigma_{n+1}^{(1)}\bigg|_{x=2n+2}=-\frac{(2n+3)(6n+13)}{4(4n+9)}\quad\text{and}\quad\tau_{n+1}^{(1)}\bigg|_{x=2n+2}=0.
\]
Therefore, it suffices to show
\begin{align*}
 & \quad-3(n+2){}^{2}\sum_{i=1}^{n+3}\frac{(2i-1)\Gamma\left(n+i+\frac{5}{2}\right)\Gamma\left(n-i+\frac{7}{2}\right)}{\Gamma(n+4-i)\Gamma(n+3+i)\left((2n+2)^{2}-(2i-1)^{2}\right)}\\
 & =-\frac{(2n+3)(6n+13)}{4}\sum_{i=1}^{n+2}\frac{(2i-1)\Gamma\left(n+i+\frac{3}{2}\right)\Gamma\left(n-i+\frac{5}{2}\right)}{\Gamma(n+3-i)\Gamma(n+2+i)\left((2n+2)^{2}-(2i-1)^{2}\right)},
\end{align*}
i.e., 
\begin{align*}
 & \quad\frac{(n+2)^{2}\pi(2n+5)\binom{4n+10}{2n+5}}{4^{2n+5}\left(-(4n+7)\right)}\\
 & =\sum_{i=1}^{n+2}\frac{(2i-1)\Gamma\left(n+i+\frac{3}{2}\right)\Gamma\left(n-i+\frac{5}{2}\right)}{(n+2-i)!(n+1+i)!\left((2n+2)^{2}-(2i-1)^{2}\right)}\\
 & \quad\times\left(\frac{(2n+3)(6n+13)}{4}-\frac{3(n+2)^{2}\left(n+i+\frac{3}{2}\right)\left(n-i+\frac{5}{2}\right)}{(n+3-i)(n+2+i)}\right).
\end{align*}
We use the WZ-method, i.e., the \textsf{fastZeil.m} \footnote{https://risc.jku.at/sw/fastzeil/}
package to show that the sum on the right-hand side is
\begin{align*}
 & \quad-\frac{(n+2)^{2}(2n+5)(4n+9)\sqrt{\pi}\Gamma\left(2n+\frac{7}{2}\right)}{4(2n+5)!}\\
 & =\frac{3(n+2)^{2}(2n+5)\Gamma\left(2n+\frac{11}{2}\right)\Gamma\left(\frac{1}{2}\right)}{\Gamma(2n+6)\left(-3(4n+7)\right)},
\end{align*}
which is exactly the left-hand side. 
\end{enumerate}
Therefore, we have proven (\ref{eq:RecD1}), which is equivalent to
(\ref{eq:Hn2k+5}). 
\end{proof}

\section{\label{sec:Vk}Results on $V_{k}$.}

Due to (\ref{eq:IandB}), in this section, we always let $x=2r+1$,
which makes $(x+1)/2=r+1$. 

We first begin with the following simple lemma.
\begin{lem}
\begin{equation}
\det V_{n}=H_{n}\left(\frac{B_{2k+1}(r+1)}{2k+1}\right)-\frac{1}{2}H_{n-1}\left(\frac{B_{2k+5}(r+1)}{2k+5}\right).\label{eq:SubstitutionVn}
\end{equation}
\end{lem}

\begin{proof}
By the cofactor expansion, we can see that 
\begin{align*}
 & \quad\det V_{n}-I_{0}H_{n-1}\left(\frac{B_{2k+5}(r+1)}{2k+5}\right)\\
 & =H_{n}\left(\frac{B_{2k+1}(r+1)}{2k+1}\right)-\left(r+\frac{1}{2}\right)H_{n-1}\left(\frac{B_{2k+5}(r+1)}{2k+5}\right),
\end{align*}
which gives the desired result. 
\end{proof}
\begin{proof}
[Proof of \eqref{eq:detVn}] Now, it is apparent to combine (\ref{eq:Hn2k+5})
and (\ref{eq:SubstitutionVn}), in order to prove (\ref{eq:detVn}).
The calculation and simplification are straightforward, but tedious,
so are omitted here. 
\end{proof}
For the rest of this section, we shall provide alternative proofs
for Props.~(\ref{prop:invertible}) and (\ref{prop:Inverse}), without
directly using (\ref{eq:detVn}). This approach eventually leads an
identity involving Stirling numbers, as Cor.~(\ref{cor:Stirling})
in the next section. 
\begin{lem}
We have 
\begin{equation}
V_{r-1}=VS(r)^{T}VS(r),\label{eq:Vandermonde}
\end{equation}
where $VS(n)$ is the $n\times n$ Vandermonde matrix of the first
$n$ squares, namely, 
\[
VS(r)=\begin{pmatrix}1 & 1 & 1 & \cdots & 1\\
1 & 2^{2} & 2^{4} & \cdots & 2^{2(r-1)}\\
1 & 3^{2} & 3^{4} & \cdots & 3^{2(r-1)}\\
\vdots & \vdots & \vdots & \ddots & \vdots\\
1 & n^{2} & n^{4} & \cdots & n^{2(r-1)}
\end{pmatrix}.
\]
\end{lem}

\begin{proof}
This directly follows from 
\[
\sum_{\ell=1}^{r}\left(\ell^{2}\right)^{k-1}\cdot\left(\ell^{2}\right)^{j-1}=\sum_{\ell=1}^{r}\ell^{2(k-1)+2(j-1)}=I_{2(k+j)-4},
\]
which is the $(k,j)$ (and also $(j,k)$) entry of $V_{r-1}$. Here,
please note that the $(1,1)$ entry is $I_{0}$, i.e., with both horizontal
and vertical indices shifted. 
\end{proof}
\begin{cor}
\label{cor:nLessr}For $n<r$, $\det V_{n}\neq0$.
\end{cor}

\begin{proof}
Note that (\ref{eq:Vandermonde}) indicates that 
\[
\det V_{r-1}=\prod_{1\leq\ell<j\leq r}(j^{2}-\ell^{2})^{2}\neq0.
\]
We see for all $n<r-1$, $V_{n}$ is one of the principal submatrices
of $V_{r-1}$, which is real, symmetric, and positive-definite, due
to again (\ref{eq:Vandermonde}). Hence $V_{n}$ is invertible. 
\end{proof}
\begin{cor}
\label{cor:nGreaterr}For $n>r$, $\det V_{k}=0$. 
\end{cor}

\begin{proof}
First of all For $n>r$ in (\ref{eq:B2k+1}), the term $(2r+1)^{2}-(2r+1)^{2}$
will appear in the product, so that 
\[
H_{n}\left(\frac{B_{2k+1}(r+1)}{2k+1}\right)=0.
\]
Similarly, for $n>r-1$, in (\ref{eq:B2k+3}), the product becomes
$0$, which also, by (\ref{eq:HankelDetLeftShifted1}), leads to
\[
H_{n}\left(\frac{B_{2k+5}(r+1)}{2k+5}\right)=0.
\]
Therefore, for $n>r$, i.e., $n-1>r-1$, $\det V_{k}=0$, by (\ref{eq:SubstitutionVn}). 
\end{proof}
Now, we can give an alternative proof of Prop.~\ref{prop:Inverse}. 
\begin{proof}
[Proof of Prop.~\ref{prop:Inverse}] Let $n=r-1$ in (\ref{eq:SubstitutionVn})
to see
\[
H_{r-2}\left(\frac{B_{2k+5}(r+1)}{2k+5}\right)=2\left(H_{r-1}\left(\frac{B_{2k+1}(r+1)}{2k+1}\right)-\det V_{r-1}\right).
\]
Meanwhile, by the formula of cofactors to compute the inverse of a
matrix, 
\[
\left(V_{r-1}^{-1}\right)_{1,1}=\frac{H_{r-2}\left(\frac{B_{2k+5}(r+1)}{2k+5}\right)}{\det V_{r-1}}=2\left(\frac{H_{r-1}\left(\frac{B_{2k+1}(r+1)}{2k+1}\right)}{\det V_{r-1}}-1\right).
\]
Therefore, it is equivalent to show
\[
\frac{H_{r-1}\left(\frac{B_{2k+1}(r+1)}{2k+1}\right)}{\det V_{r-1}}=\frac{\binom{4r}{2r}}{\binom{2r}{r}^{2}},
\]
which can be done by simplifying the left-hand side. Details are omitted
here. 
\end{proof}
Finally, we shall complete the proof of Prop.~\ref{prop:invertible}.
By Cors.~\ref{cor:nLessr} and \ref{cor:nGreaterr}, we only need
to show $\det V_{r}=0$. The following sequence plays the essential
rule. 
\begin{defn}
The sequence $T(n,k)$ (A204579\footnote{https://oeis.org/A204579})
can be defined by the recurrence: 
\begin{equation}
T(n,k)=\begin{cases}
1, & n=k=1;\\
T(n-1,k-1)-(n-1)^{2}T(n-1,k), & n>1,1\leq k\leq n;\\
0, & \text{otherwise.}
\end{cases}\label{eq:Tnk}
\end{equation}
\end{defn}

\begin{lem}
\label{lem:EigenVector}Let $\mathbf{0}$ be the zero vector. We have
\begin{equation}
V_{r}\cdot\left(\begin{array}{cccc}
T(r+1,1) & T(r+1,2) & \cdots & T(r+1,r+1)\end{array}\right)^{T}=\mathbf{0}.\label{eq:ZeroVector}
\end{equation}
\end{lem}

\begin{proof}
First of all, (\ref{eq:ZeroVector}) is equivalent to the following
$r+1$ identities: for any $j=0,1,\ldots,r$, 
\begin{equation}
\sum_{k=0}^{r}I_{2j+2k}T(r+1,k+1)=0.\label{eq:ZeroVectorExpansion}
\end{equation}
By definition and simplification, we have
\[
\sum_{k=0}^{r}I_{2j+2k}T(r+1,k+1)=\sum_{c=1}^{r}c^{2j}\sum_{k=0}^{r}c^{2k}T(r+1,k+1).
\]
Now, we claim, for $c=1,2,\ldots,r$, 
\begin{equation}
\sum_{k=0}^{r}c^{2k}T(r+1,k+1)=0,\label{eq:ZeroVectorInner}
\end{equation}
which implies (\ref{eq:ZeroVectorExpansion}), for any $j$. 

When $r=1$, the only case is $c=1$. Since $T(2,1)=-1$, and $T(2,2)=1,$we
have (\ref{eq:ZeroVectorInner}) for $r=1$. Assume (\ref{eq:ZeroVectorInner})
holds for $r=m-1$. For $r=m$, by the recurrence (\ref{eq:Tnk}),
we have
\begin{align*}
\sum_{k=0}^{m}c^{2k}T(m+1,k+1) & =\sum_{k=0}^{m}c^{2k}\left(T(m,k)-m^{2}T(m,k+1)\right)\\
 & =\sum_{k=0}^{m}c^{2k}T(m,k)-m^{2}\sum_{k=0}^{m}c^{2k}T(m,k+1)\\
 & =\sum_{k=1}^{m}c^{2k}T(m,k)-m^{2}\sum_{k=0}^{m-1}c^{2k}T(m,k+1),
\end{align*}
where in the last step, we used $T(m,0)=T(m,m+1)=0$. Finally, by
shifting the summation index of the first sum, we have 
\begin{align*}
\sum_{k=0}^{m}c^{2k}T(m+1,k+1) & =c^{2}\sum_{k=0}^{m-1}c^{2k}T(m,k+1)-m^{2}\sum_{k=0}^{m-1}c^{2k}T(m,k+1)\\
 & =(c^{2}-m^{2})\sum_{k=0}^{m-1}c^{2k}T(m,k+1)=0,
\end{align*}
for $c=0,1,\ldots,m-1$, (by the inductive assumption) and $c=m$
(due to the first factor). 
\end{proof}
\begin{lem}
\label{lem:Vr}$\det V_{r}=0$, for all $r\in\mathbb{N}$. 
\end{lem}

\begin{proof}
Note that $T(m,m)=T(m-1,m-1)=\cdots=T(1,1)=0$, so the vector in (\ref{eq:ZeroVector})
is not a zero vector. Therefore $V_{r}$ is not invertible. 
\end{proof}

\section{\label{sec:FinalRemarks}Final remarks}

\subsection{An identity on Stirling numbers.}

Let $s(n,k)$ be the Stirling numbers of the first kind and $S(n,k)$
be the Stirling numbers of the second kind. $T(n,k)$ also has an
alternative expression: 
\[
T(r+1,k+1)=\sum_{i=0}^{2k+2}(-1)^{r+1+i}s(r+1,i)s(r+1,2k+2-i).
\]
Meanwhile, recall the well-known connection between Bernoulli numbers
and $S(n,k)$: 
\[
B_{m}=\sum_{\ell=0}^{m}\frac{(-1)^{\ell}\ell!}{\ell+1}S(m,\ell).
\]
(See \cite[Entry (24.15.6)]{NIST}.) Then, (\ref{eq:IandB}) indicates
\[
I_{2j+2k}=\frac{1}{2j+2k+1}\sum_{m=0}^{2j+2k}\binom{2j+2k+1}{m}r^{2j+2k+1-m}\sum_{\ell=0}^{m}\frac{(-1)^{\ell}\ell!}{\ell+1}S(m,\ell),
\]
where the expansion 
\[
B_{n}(x)=\sum_{m=0}^{n}\binom{n}{m}x^{n-m}B_{m}
\]
(see \cite[Entry (24.2.5)]{NIST}) is also applied. Therefore, (\ref{eq:ZeroVectorExpansion})
yields the following identity of Stirling numbers. 
\begin{cor}
\label{cor:Stirling}For any $r\in\mathbb{N}$ and $j=0,1,\ldots,r$,
we have 
\begin{align*}
\sum_{k=0}^{r}\frac{1}{2j+2k+1} & \left(\sum_{m=0}^{2j+2k}\binom{2j+2k+1}{m}((r+1)^{2j+2k+1-m}-1)\sum_{\ell=0}^{m}\frac{(-1)^{\ell}\ell!}{\ell+1}S(m,\ell)\right.\\
 & \times\left.\sum_{i=0}^{2k+2}(-1)^{r+1+i}s(r+1,i)s(r+1,2k+2-i)\right)=0.
\end{align*}
\end{cor}

\subsection{Continued fraction approach in general }

\begin{onehalfspace}
\noindent It is natural to consider the continued fraction to the
generating function of $B_{2k+5}(\tfrac{x+1}{2})/(2k+5)$: 

\noindent 
\[
H(z)=\sum_{n=0}^{\infty}\frac{B_{2n+5}(\frac{1+x}{2})}{2n+5}z^{2n}=\frac{G(z)-B_{3}\left(\frac{1+x}{2}\right)}{z^{2}}=\frac{G(z)-\frac{x^{3}-x}{24}}{z^{2}}
\]
Then, 
\[
\frac{z^{2}H(z)}{\frac{x^{3}-x}{24}}=\frac{G(z)}{\frac{x^{3}-x}{24}}-1.
\]
In the proof of (\ref{eq:B2k+3}), we actually have shown that 
\[
\frac{G(z)}{\frac{x^{3}-x}{24}}=\frac{1}{1-(\alpha_{1}+\alpha_{2})z^{2}-\frac{\alpha_{2}\alpha_{3}z^{4}}{1-(\alpha_{3}+\alpha_{4})z^{2}-\frac{\alpha_{4}\alpha_{5}z^{4}}{1-\ddots}}}.
\]
We then define another sequence $(\beta_{n})_{n\geq1}$ by $\beta_{1}=\alpha_{1}+\alpha_{2}$,
and for $m\geq1$, 
\begin{align}
\beta_{2m-1}\beta_{2m} & =\alpha_{2m}\alpha_{2m+1},\label{eq:DEFBetaN1}\\
\beta_{2m}+\beta_{2m+1} & =\alpha_{2m+1}+\alpha_{2m+2}.\label{eq:DEFBetaN2}
\end{align}
This not only allows us to recursively solve $\beta_{n}$; but also
indicate, by the inverse even contraction (\ref{eq:EvenContraction})
and the odd contraction (\ref{eq:OddContraction}), that
\[
\frac{G(z)}{\frac{x^{3}-x}{24}}=\frac{1}{1-\frac{\beta_{1}z^{2}}{1-\frac{\beta_{2}z^{2}}{1-\ddots}}}=1+\frac{\beta_{1}z^{2}}{1-(\beta_{1}+\beta_{2})z^{2}-\frac{\beta_{2}\beta_{3}z^{4}}{1-(\beta_{3}+\beta_{4})z^{2}-\frac{\beta_{4}\beta_{5}z^{4}}{\ddots}}},
\]
which eventually leads to 
\[
H(z)=\frac{\frac{x^{3}-x}{24}\beta_{1}}{1-(\beta_{1}+\beta_{2})z^{2}-\frac{\beta_{2}\beta_{3}z^{4}}{1-(\beta_{3}+\beta_{4})z^{2}-\frac{\beta_{4}\beta_{5}z^{4}}{\ddots}}}.
\]
Note that 
\[
\frac{x^{3}-x}{24}\beta_{1}=\frac{x^{3}-x}{24}\left(\frac{x^{2}-1}{12}+\frac{x^{2}-4}{15}\right)=\frac{(3x^{2}-7)x(x^{2}-1)}{480}=\frac{B_{5}\left(\frac{x+1}{2}\right)}{5}.
\]
Hence, 
\[
H_{n}\left(\frac{B_{2k+5}\left(\frac{x+1}{2}\right)}{2k+1}\right)=\left(\frac{(3x^{2}-7)x(x^{2}-1)}{480}\right)^{n+1}\prod_{\ell=1}^{n}\left(\beta_{2\ell}\beta_{2\ell+1}\right)^{n+1-\ell}.
\]
In fact, it is not hard to see $D_{n}^{(1)}=\prod_{\ell=0}^{n}\beta_{2\ell+1}$. 
\end{onehalfspace}

\subsection{Further results}

We list some nice partial results with $x=2r+1$ that can be easily
calculated from the results above.
\begin{itemize}
\item $\alpha_{2r+1}=0$, which implies $\tau_{r}^{(1)}=0$.
\item $D_{r-1}^{(1)}=r!^{2}$ and $D_{r}^{(1)}=-\frac{(r+1)!^{2}}{4r+5}.$
\item $\beta_{2r}=0$ and $\beta_{2r+1}=-\frac{(r+1)^{2}}{4r+5}$.
\item 
\[
\frac{H_{r}\left(\frac{B_{2k+1}(r+1)}{2k+1}\right)}{H_{r-1}\left(\frac{B_{2k+5}(r+1)}{2k+5}\right)}=2\quad\text{and}\quad\frac{H_{r-1}\left(\frac{B_{2k+5}(r+1)}{2k+5}\right)}{H_{r-1}\left(\frac{B_{2k+3}(r+1)}{2k+3}\right)}=\left(r!\right)^{2}.
\]
\end{itemize}

\section*{acknowledgment}

We would like to thank Dr.~Christian Krattenthaler for providing
the expressions of (\ref{eq:Hn2k+5}) and (\ref{eq:detVn}), right
after we uploaded the first version of this paper on arXiv. Also,
we thank Dr.~Dongmian Zou for his general idea on considering the
eigenvector with respect to $0$, as the key in the proof of Lem.~\ref{lem:EigenVector},
which eventually leads to Corol.~\ref{cor:Stirling}.


\begin{thebibliography}{10}
\bibitem{Euler}W.~A.~Al-Salam and L.~Carlitz, Some determinants
of Bernoulli, Euler and related numbers, \emph{Portugal. Math.} \textbf{18}
(1959), 91--99.

\bibitem{RamanujanII}B.~C.~Berndt, \emph{Ramanujan\textquoteright s
Notebooks, Part II}. Springer-Verlag, New York 1989. 

\bibitem{Chi}T.~S.~Chihara, \textit{An Introduction to Orthogonal
Polynomials}, Gordon and Breach, 1978.

\bibitem{CF}A.~Cuyt, V.~B.~Petersen, B.~Verdonk, H.~Waadeland,
and W.~B.~Jones, \emph{Handbook of Continued Fractions for Special
Functions}. Springer, New York, 2008.

\bibitem{Wenlin}W.~Dai, X.~Tong, and T.~Tong, Optimal-$k$ sequence
for difference-based methods in nonparametric regression, in preparation. 

\bibitem{Inverse}H.~Dettet, A.~Munk and T.~Wagner, Estimating
the variance in nonparametric regression-what is a reasonable choice?
\emph{J.~R.~Statist.~Soc.~B} \textbf{60} (1998) 60, 751--764. 

\bibitem{DJ1}K.~Dilcher and L.~Jiu, Orthogonal polynomials and
Hankel determinants for certain Bernoulli and Euler polynomials, \emph{J.~Math.~Anal.~Appl}.~\textbf{497}
(2021), Article 124855.

\bibitem{DJ2}K.~Dilcher and L.~Jiu, Hankel determinants of sequences
related to Bernoulli and Euler polynomials. To Appear in \emph{Int.~J.~Number
Theory}, Preprint, 2020, arXiv:2105.01880. 

\bibitem{DJ3}K.~Dilcher and L.~Jiu, Hankel Determinants of shifted
sequences of Bernoulli and Euler numbers, Preprint, 2021, arXiv:2007.09821

\bibitem{ErMaObTr}A~Erd\'{e}lyi, W.~Magnus, F.~Oberhettinger,
and F.~G.~Tricomi, \emph{Higher Transcendental Functions, Vol.~I}.
Based, in part, on notes left by Harry Bateman. McGraw-Hill, 1953.

\bibitem{FK1}M. Fulmek and C. Krattenthaler, The number of rhombus
tilings of a symmetric hexagon which contain a fixed rhombus on the
symmetry axis,II, \emph{Ann.~Combin.}~\textbf{2} (1998), 19--40.

\bibitem{FK2}M. Fulmek and C. Krattenthaler, The number of rhombus
tilings of a symmetric hexagon which contain a fixed rhombus on the
symmetry axis, II, \emph{European J.~Combin.}~\textbf{21} (2000),
601--640.

\bibitem{Han} G.-N.~Han, Jacobi continued fraction and Hankel determinants
of the Thue-Morse sequence, \textit{Quaest. Math.} \textbf{39} (2016),
895--909.

\bibitem{Is} M.~E.~H.~Ismail, \textit{Classical and Quantum Orthogonal
Polynomials in One Variable. With Two Chapters by Walter Van Assche}.
Encyclopedia of Mathematics and its Applications, 98. Cambridge University
Press, Cambridge, 2005.

\bibitem{NIST}F.~W.~J.~Olver et al. (eds.), \emph{NIST Handbook
of Mathematical Functions}, Cambridge Univ.~Press, New York, 2010.
Online version: http://dlmf.nist.gov.

\bibitem{Kr1}C.~Krattenthaler, Advanced determinant calculus, \emph{S\'{e}minaire
Lotharingien Combin}.~\textbf{42} (``The Andrews Festschrift'') (1999),
Article B42q.

\bibitem{Kr2}C.~Krattenthaler, Advanced determinant calculus: a
complement. \emph{Linear Algebra Appl.} \textbf{411} (2005), 68--166. 

\bibitem{Mi}S.~C.~Milne, Infinite families of exact sums of squares
formulas, Jacobi elliptic functions, continued fractions, and Schur
functions, \emph{Ramanujan J}. \textbf{6} (2002), 7--49.

\bibitem{CF2}L.~Lorentzen and H.~Waadeland, \emph{Continued Fractions
with Applications}. Studies in Computational Mathematics, 3. North-Holland,
Amsterdam, 1992.

\bibitem{MWY}L.~Mu, Y.~Wang, and Y.~Yeh, Hankel determinants of
linear combinations of consecutive Catalan-like numbers, \textit{Discrete
Math.} \textbf{340} (2017), 3097--3103.

\bibitem{CFWall}H.~S.~Wall, \emph{Analytic Theory of Continued
Fractions}. Van Nostrand, New York, 1948.
\end{thebibliography}
\end{document}